\pdfoutput=1
\documentclass[preprint,svgnames,dvipsnames]{elsarticle}

\usepackage[OT1]{fontenc}
\usepackage{amsmath,amsfonts, amssymb,amsthm}
\usepackage[numbers]{natbib}
\usepackage[colorlinks,citecolor=blue,urlcolor=blue]{hyperref}
\usepackage[left=3cm,right=3cm,top=3cm,bottom=3cm]{geometry}
\usepackage{algorithm}
\usepackage{algorithmicx}
\usepackage{algpseudocode}
\usepackage{enumerate}
\usepackage{bm}
\usepackage{graphicx}
\usepackage{dsfont}
\usepackage{tabularx}
\usepackage{stmaryrd}
\usepackage[svgnames,dvipsnames]{xcolor}
\usepackage{listings}
\usepackage{mathtools}

\newtheorem{theo}{Theorem}
\newtheorem{proposition}{Proposition}
\newtheorem{lme}{Lemma}
\newtheorem{coro}{Corollary}
\newtheorem{defi}{Definition}
\newtheorem{rema}{Remark}

\newcommand{\R}{\mathbb{R}}
\newcommand{\N}{\mathbb{N}}

\newcommand{\A}{\mathbb{A}}
\newcommand{\calS}{\mathcal{S}}
\newcommand{\prob}{\mathbf{P}}
\newcommand{\proba}{\mathbb{P}}
\newcommand{\V}[1]{\mathbb{V}\left( #1 \right)}
\newcommand{\E}[1]{\mathbb{E}\left[ #1 \right]}
\newcommand{\Cov}[2]{{\normalfont \textrm{Cov}}\left( #1 , #2 \right)}
\newcommand{\Card}[1]{\left\lvert #1 \right\rvert}
\newcommand{\diffCard}[2]{\Card{#1} - \Card{#2}}
\newcommand{\pset}[1]{\mathcal{P}\left(#1\right)}
\newcommand{\inprodL}[1]{\left\langle #1 \right\rangle_{L^2}}
\newcommand{\normL}[2]{\left\lVert #1 \right \rVert_{#2}}
\newcommand{\krnl}[2]{k\left( #1, #2 \right)}
\newcommand{\SMMD}{S^{\text{MMD}}}

\newcommand{\ie}{i.e.,~}

\newcommand{\cororef}[1]{Corollary~\ref{#1}}
\newcommand{\defref}[1]{Definition~\ref{#1}}
\newcommand{\propref}[1]{Proposition~\ref{#1}}
\newcommand{\lmeref}[1]{Lemma~\ref{#1}}

\makeatletter
\def\ps@pprintTitle{%
\let\@oddhead\@empty
\let\@evenhead\@empty
\def\@oddfoot{}%
\let\@evenfoot\@oddfoot}
\makeatother

\begin{document}

\begin{frontmatter}

%%%%%%%%%%%%%%%%%%%%%%%%%%%%%%%%
% Title & runtitle
\title{On the coalitional decomposition of parameters of interest}

%%%%%%%%%%%%%%%%%%%%%%%%%%%%%%%%
% Authors information
\author[a,b,c,e]{Marouane Il Idrissi}
\author[a,b,d]{Nicolas Bousquet}
\author[c]{Fabrice Gamboa}
\author[a,b,c]{Bertrand Iooss}
\author[c]{Jean-Michel Loubes}

\address[a]{EDF Lab Chatou, 6 Quai Watier, 78401 Chatou, France}
\address[b]{SINCLAIR AI Lab., Saclay, France}
\address[c]{Institut de Mathématiques de Toulouse, 31062 Toulouse, France}
\address[d]{Sorbonne Université, LPSM, 4 place Jussieu, Paris, France}
\address[e]{Corresponding Author - Email: marouane.il-idrissi@edf.fr}

\begin{abstract}
  Understanding the behavior of a black-box model with probabilistic inputs can be based on the decomposition of a parameter of interest (e.g., its variance) into contributions attributed to each coalition of inputs (i.e., subsets of inputs). In this paper, we produce conditions for obtaining unambiguous and interpretable decompositions of very general parameters of interest. This allows to recover known decompositions, holding under weaker assumptions than stated in the literature.
\end{abstract}


\end{frontmatter}

%%%%%%%%%%%%%%%%%%%%%%%%%%%%%%%%%%%%%%%%%%%%%%%%%%%%%%%%%%%%%%%%%%%%%%%%%%
%% INTRODUCTION & DEFINITIONS
\section{Introduction and preliminaries}

The decomposition of a parameter of interest, also known as a quantity of interest (QoI) in the uncertainty quantification framework, with respect to (w.r.t.) coalitions of covariables is crucial in both the field of sensitivity analysis of numerical models and in explainable artificial intelligence \cite{Iooss2022}. These decompositions allow to distribute shares of QoI to the inputs of an input-output black-box model. Depending on the QoI, they both allow to better understand the behavior of such models, and to perform post-hoc interpretability \cite{Barredo2020}.

For instance, the well-known Hoeffding-Sobol' decomposition is a particular instance of output variance decomposition, which has been used for both settings \cite{davgam21, Fel2021, Benesse2022}. It relies on a unique decomposition of an input-output model in $L^2$. Nevetheless, it requires independent covariables \cite{Hoeffding1948}, but allows to quantify the influence (in terms of percentages of output variance) of each inputs of a black-box model, as well as interaction influence due to coalitions of inputs.

In this paper, the concept of ``coalitional decomposition of QoI'' is developped, generalizing the idea of the Hoeffding-Sobol' variance decomposition to other types of QoIs, leveraging results from the field of combinatorics. In particular, Rota's extension of the Möbius inversion formula to partially ordered sets \cite{CombiRota2012}). Necessary conditions are presented in order to define coalitional decompositions of abstract QoIs. It is shown, among other QoI decompositions proposed in the litterature, that the Hoeffding-Sobol' decomposition still holds without the need for independent inputs, but its interpretation as interaction effects holds only when input independence is assumed. Furthermore, a quite general point of view is adopted, allowing to define decompositions for a large variety of QoIs.
%%%%%%%%%%%%%%%%%%%%%%%%%%%%%%%%%%%%%%%%%%
%% INGREDIENTS
\subsection{Notations and tools}
\subsubsection{Inputs, model and outputs}
Let $(\Omega, \mathcal{F}, \prob)$ be some probability space. Let, for $i=1,\dots, d$, $d \in \N^*$, $\left(E_i, \mathcal{B}(E_i)\right)$ be abstract polish measurable space, i.e., $E_i$ is a separable completely metrizable topological space, and $\mathcal{B}(E_i)$ denotes its associated Borel $\sigma$-algebra. Let $D=\{1,\dots,d\}$ and denote by $\pset{D}$ its power-set (i.e., the set of all possible subsets of $D$, including $\emptyset$). For any $A \subseteq D$, denote the \emph{marginal} measurable spaces $\left(E_A, \mathcal{E}_A \right)$, where
$$\quad E_A = \bigtimes_{i \in A} E_i, \quad \mathcal{E}_A = \bigotimes_{i\in A} \mathcal{B}(E_i) = \mathcal{B}\left( \bigtimes_{i \in A} E_i\right),$$
Set also $(E, \mathcal{E}) := (E_D, \mathcal{E}_D)$. Let $X = (X_1, \dots, X_d)^\top$ be an $E$-valued random vector (i.e., a measurable function $X : \Omega \rightarrow E$), referred to as the \emph{inputs}. Let $P_X$ be the distribution of the inputs. Define the marginal distributions, for each $A \subset D$, as:
$$P_{X_A} = \int_{E_{\overline{A}}} dP_X,$$
where $X_A=(X_i)_{i\in A}$ is the coalition of inputs whose indices are in $A$ (i.e., the subset $X_A$ of $X$). Further, $\overline{A}$ denotes the complementary set of $A$ in $D$ (\ie $\overline{A} = D \setminus A$). Additionally, for every $A \subset D$, the conditional distributions $P_{X_A \mid X_{\overline{A}}}$ are assumed to be regular, and if not uniquely defined, they are chosen to be regular (see \cite{Breiman1992}, Chap. 4).

Let $G: E \rightarrow Z$ be an measurable function. Here $Z$ denotes an abstract polish space. $G(X)$ is the $Z$-valued random variable, resulting from the composition of $G$ with $X$. In the following, the function $G$ is referred to as a \emph{model}, meanwhile $G(X)$ is referred to as the \emph{output} of the model. Denote $\proba(E)$ the set of all probability distributions on $(E, \mathcal{E})$. $\mathcal{M}(E)$ denotes the set of $Z$-valued models, i.e., every $Z$-valued, measurable functions. 

\begin{rema}
In essence, the random inputs $X$, and the output $G(X)$ are not restricted to be real-valued, but can be defined on more complex measurable spaces (e.g., images, functions, stochastic processes).
%, as long as it is possible to formally define probability measures on such objects (hence the polish condition on $E$).
\end{rema}

A particular subset of $\mathcal{M}(E)$ is of interest in the present work whenever $Z = \R$: $L^2(P_X, \R)$. It is the set of measurable, $\R$-valued functions which are square-integrable against $P_X$. Recall that $L^2(P_X; \R)$ is a Hilbert space with the inner product:
$$\forall f,g \in L^2(P_X; \R), \quad \inprodL{f,g} = \int_E f(x)g(x) dP_X(x),$$
and associated norm:
$$\forall f \in L^2(P_X; \R), \quad \normL{f}{L^2}^2 = \int_E f^2(x) dP_X(x).$$
Denote, for any $A \subset D$, $L^2\left(P_{X_A}; \R \right)$ the Hilbert subspaces of $L^2(P_X ; \R)$, of square integrable, $\mathcal{E}_A$-measurable functions. In other words, any $f \in L^2(P_{X_A}, \R)$ is a square-integrable function $f : E_A \rightarrow \R$: elements of $L^2(P_{X_A}, \R)$ only take $|A|:=\textrm{card}(A)$ inputs. Whenever $Z = \R^k$ for a positive integer $k$, one can also define the set $L^2\left(P_X; \R^k\right)$ accordingly (see \cite{Gamboa2013}).

\subsubsection{Some elements of combinatorics and abstract algebra}
A \emph{partially ordered set} (poset) is defined as a pair $(\calS, \leq)$ where $\mathcal{S}$ is a non-empty set, and $\leq$ is a partial order binary relation on elements of $\calS$. A poset $(\calS, \leq)$ is said to be \emph{locally finite} if, for any $x,z \in \calS$, the sets $\{y \in \calS : x \leq y \leq z \}$ (also called \emph{segments} of $\mathcal{S}$) are finite.

A \emph{commutative ring with identity}, is a triplet $(\A, +, \times)$ where $\A$ is a non-empty set, and where $+$ and $\times$ are addition and multiplication operators respectively, which are both associative and commutative on $\A$, $\times$ is distributive w.r.t. $+$ on $\A$, $\A$ contains both an additive and multiplicative identity, but only an additive inverse. A commutative ring with identity that admits a multiplicative inverse is generally called a \emph{field}. In the following, abstract commutative rings with identity are denoted $\A$, and are assumed to be endowed with the usual addition and multiplication operator, unless stated otherwise. For instance, $\R$ is a commutative ring with identity (it is in fact, a field).

Denote $I_{\A}(\calS)$ the incidence algebra of a locally finite poset $(\calS, \leq)$ over a commutative ring with identity $\A$, i.e., the set of functions $f:\calS \times \calS \rightarrow \A$ such that $f(x,y)=0$ if $x \not \leq y$ (see \cite{Spiegel1997}, Definition~1.2.1 p.10). $(I_{\A}(\calS), +, *)$ forms an $\A$-algebra with the usual pointwise addition $+$ and the usual convolution $*$, i.e., for any $f,g \in I_{\A}(\calS)$, and any $x,z \in \mathcal{S}$ such that the segment $\{y \in \mathcal{S} : x \leq y \leq z\}$ is non-empty,
$$(f*g)(x,z) = \sum_{x \leq y \leq z}f(x,y)g(y,z).$$
The zeta function $\zeta \in I_{\A}(\calS)$ is convolutional identity of the incidence algebra, and is defined as, $\forall x,y \in \calS$:
$$\zeta(x,y) = \begin{cases} 
1 & \text{if } x=y,\\
0 & \text{otherwise.}
\end{cases}$$
The Möbius function, denoted $\mu \in I_{\A}(\calS)$, in the case of locally finite posets $\calS$, is defined as the \emph{inverse of the zeta function for the convolution operator} defined on the incidence algebra of $\calS$, and can be computed recursively, for any $x,y \in \calS$ with $x \leq y$, as \cite{Kock2020}
$$\mu(x,y)=\begin{cases}1 & \text{if } x=y \\ \displaystyle - \sum_{x \leq z < y} \mu(x,z)  & \text{otherwise.}\end{cases}$$

Finally, in the scope of this work, it is important to note that, for the finite set $D$, the pair $(\pset{D}, \subseteq)$ where $\subseteq$ denotes the inclusion between sets, forms an locally finite poset.

%%%%%%%%%%%%%%%%%%%%%%%%%%%%%%%%%%%%%%%%%%
%% FORMAL QoI DEFINITION
\subsection{Quantity of interest}

A QoI (or parameter of interest) is the mapping of a model $G \in \mathcal{M}(E)$ and an input distribution $P_X \in \proba(E)$ to a commutative ring with identity $\A$. They can be formally defined as follows:
\begin{defi}[Quantity of interest]
An $\A$-valued QoI on a model $G$ with random inputs $X \sim P_X$, is an application:
\begin{align*}
    \phi : \proba(E) \times \mathcal{M}(E)  &\rightarrow \A \\
    P \times H  &\mapsto \phi_{P}(H).
\end{align*}
onto $G$ and $P_X$, i.e., $\phi_{P_X}(G)$.
\end{defi}
Whenever $Z = \A = \R$, for inputs $X \sim P_X$ and a model $G \in L^2(P_X; \R)$, an example of a QoI on $G$ and $X$ can be the variance of the output:
\begin{align*}
    \phi_{P_X}(G) &= \int_E \left( G(x) - \int_{E} G(t) dP_X(t) \right)^2 dP_X(x) \\
    &= \normL{G - \E{G(X)}}{L^2}^2 \\
    &= \V{G(X)}
\end{align*}
Other examples of QoIs on $G$ can be its generalized moments w.r.t. $X$, the probability that $G(X)$ exceeds a fixed threshold, or a quantile of $G(X)$ given a certain level. This definition of a QoI is very general on purpose. In essence, QoIs can also be random variables. However, for the sake of simplicity, in the remainder of this work, it is assumed that for any model $G$ with inputs $X \sim P_X$, $\phi_{P_X}(G)$ is not random.

%%%%%%%%%%%%%%%%%%%%%%%%%%%%%%%%%%%%%%%%%%
%% MÖBIUS INVERSION FORMULA & INCLU-EXCLU
\subsection{Möbius inversion formula and the Inclusion-Exclusion principle}

Originally, the ``classic'' Möbius inversion formula has been first discovered in the field of number theory by \cite{Mobius1832}. It provides a particular relation between pairs of arithmetic functions (i.e., defined on the natural numbers). This result has since been extended to locally finite posets, and became one of the main foundational result in the field of combinatorics \cite{Rota1964}. This extension, as stated in \cite{CombiRota2012} (Section~3.1.2 p.108) writes as follows:
\begin{theo}[Möbius inversion formula on locally finite posets]
Let $\calS$ be any non-empty set and $(\calS, \leq)$ form a locally finite poset, where $\leq$ is a binary relation. Let $\varphi$ and $\psi$ be functions from $\calS$ to $\A$. Then, the following equivalence hold:
$$\varphi(x) = \sum_{y:y \leq x} \psi(y), \quad \forall x \in \calS \quad \iff \quad \psi(x) = \sum_{y:y\leq x} \varphi(y) \mu(y,x), \quad \forall x \in \calS.$$
where $\mu$ is the Möbius function.
\end{theo}

The Möbius function, for certain particular posets, admit a closed form. In particular, on the locally finite poset formed by $(\pset{D}, \subseteq)$, for any $B \subseteq A \in \pset{D}$, the Möbius function writes (see \cite{Rota1964}, Corollary p.345):
$$\mu(B, A) = (-1)^{\diffCard{A}{B}}.$$
It comes from the fact that the poset $(\pset{D}, \subseteq)$ is a Boolean lattice \cite{Rota1964}. It leads to the following result (see \cite{CombiRota2012} Section~3.1.1 p.108).
\begin{coro}[Möbius inversion formula on power-sets]
Let $\varphi$ and $\psi$ be functions from $\pset{D}$ to $\A$. Then the following equivalence holds:
$$\varphi_A = \sum_{B\subseteq A} \psi_B, \quad \forall A \in \pset{D} \quad \iff \quad \psi_A = \sum_{B \subseteq A} (-1)^{\diffCard{A}{B}} \varphi_B, \quad \forall A \in \pset{D}.$$
\label{coro:mobInvGen}
\end{coro}
\cororef{coro:mobInvGen} can be seen as a generalization of the \emph{Inclusion-Exclusion principle}. It allows to decompose additive functions $f : \calS \rightarrow \mathbb{R}$ where $\calS$ is an algebra of sets. It is widely used in probability theory. However, in light of this generalization, two main differences arise when compared to the classical principle:
\begin{itemize}
    \item Both statements in \cororef{coro:mobInvGen} are equivalent, whereas for the classical Inclusion-Exclusion principle, the left-hand statement only implies the right-left statement;
    \item The functions to be decomposed are not restricted to be additive, and valued in $\R$ (or even a field) anymore, but they must only be valued in a commutative ring with identity (or even, in some cases, an Abelian group).
\end{itemize}
The consequences of these differences allow to easily define \emph{coalitional QoI decompositions}, for a broad range of QoIs, and with minimal assumptions on the model $G$ and the distribution $P_X$ of its inputs.

%%%%%%%%%%%%%%%%%%%%%%%%%%%%%%%%%%%%%%%%%%%%%%%%%%%%%%%%%%%%%%%%%%%%%%%%%%
%% COALITIONAL DECOMPOSITIONS
\section{Coalitional decompositions of QoIs}

A \emph{coalition} of inputs indexed by $A \in \pset{D}$ refers to the subset of $E_A$-valued random inputs $X_A$. In its essence, a coalitional QoI decomposition amounts to writing a QoI as a sum of terms indexed by a set $A \in \pset{D}$, relative to each subset $X_A$ of inputs. They can be formally defined as follows.
\begin{defi}[Coalitional decompositions]
Let $G \in \mathcal{M}(E)$ be a model with $E$-valued random inputs $X \sim P_X \in \proba(E)$, and $\phi_{P_X}(G)$ be an $\A$-valued QoI. One says that a QoI $\phi_{P_X}(G)$ admits a coalitional decomposition if it can be written as:
$$\phi_{P_X}(G) = \sum_{A \in \pset{D}} \psi_A $$
where $\psi : \pset{D} \rightarrow \A$. The right hand side is referred to as the coalitional decomposition of $\phi_{P_X}(G)$.
\label{dfi:coalDecomp}
\end{defi}

\subsection{Main result}
It is important to note that there exists infinitely many coalitional decompositions for a fixed QoI. However, the following result leverages \cororef{coro:mobInvGen} in order to characterize a particular class of coalitional decompositions. Sufficient conditions on $\psi$ are given in order to ensure a coalitional QoI decomposition. Notice that it remarkably involves very limited assumptions on the probabilistic structure $P_X$ and the model $G$.

\begin{lme}[Möbius decomposition]
Let $G \in \mathcal{M}$ a model with $E$-valued random inputs $X \sim P_X \in \proba(E)$. Let $\phi_{P_X}(G)$ be a QoI on $G$. Let $\varphi : \mathcal{P}(D) \rightarrow \A$ be a set function such that:
$$\varphi_D = \phi_{P_X}(G).$$
and $\forall A \in \pset{D}, \varphi_A$ is well-defined. Then, $\phi_{P_X}(G)$ admits the following coalitional decomposition:
\begin{equation}
    \phi_{P_X}(G) = \sum_{A \in \mathcal{P}(D)} \psi_A,
    \label{eq:mobPsi}
\end{equation}
where, $\forall A \subseteq D$,
\begin{equation}
    \psi_A = \sum_{B \subseteq A} (-1)^{\diffCard{A}{B}} \varphi_B.
    \label{eq:mobPhi}
\end{equation}
\label{thm:mobDecomp}
\end{lme}
This particular characterization of the coalitional decomposition of $\phi_{P_X}(G)$ is referred to as its Möbius decomposition.
\begin{proof}[Proof of \lmeref{thm:mobDecomp}]
Since, by assumption, $\varphi_A$ is well defined $\forall A \in \pset{D}$, let:
$$ \psi_A = \sum_{B \subseteq A} (-1)^{\diffCard{A}{B}} \varphi_B, \quad \forall A \in \pset{D}.$$
By \cororef{coro:mobInvGen}, it is equivalent to:
$$\varphi_A = \sum_{B \subseteq A} \psi_B, \quad \forall A \in \pset{D},$$
and, in particular:
$$\varphi_D = \phi_{P_X}(G) = \sum_{A \in \pset{D}} \psi_A.$$
\end{proof}

One can notice from \lmeref{thm:mobDecomp} that, defining a coalitional QoI decomposition amounts to choosing a set function $\varphi$ such that $\varphi_D = \phi_{P_X}(G)$, with very limited assumptions on both $G$ and the inputs' probability structure $P_X$ (i.e., the well-definition of $\varphi_A$, $\forall A \in \pset{D}$).

\subsection{Desirable coalitional decomposition properties}
Some Möbius decompositions can be trivial: take, for instance, $\varphi_A = \phi_{P_X}(G), \forall A \in \pset{D}$. In those cases, even if the decomposition hold, $\psi_A$ is not \emph{meaningful}, in the sense that it is not related to the subset of inputs $X_A$. Hence, to ensure the meaningfulness of a Möbius decompositions, some properties can be desired, as detailed in the following.
\begin{defi}[Gradual Möbius decomposition]
Let $G \in \mathcal{M}(E)$ be a model with $E$-valued random inputs $X \sim P_X \in \proba(E)$, and let $\phi_{P_X}(G)$ be an $\A$-valued QoI on $G$. Assume that this QoI admits a Möbius decomposition (i.e., it can be written as \eqref{eq:mobPsi} with \eqref{eq:mobPhi}). If $\varphi$ can be written, for any $A \in \pset{D}$, as:
$$\varphi_A = \phi_{P_X}(f_A),$$
where $f_A \in \mathcal{M}(E_A)$ is a $Z$-valued $E_A$-measurable function, then the decomposition is said to be gradual.
\end{defi}
The term gradual refers to the functions $f_A$, whose input dimension is increasing with the cardinal of $A \in \pset{D}$. It ensures that each $\varphi_A$ is somewhat linked to the coalition of inputs $X_A \sim P_{X_A}$ through the functions $f_A$, and subsequently, $\psi_A$ as well.

While graduality ensures a link between each $\psi_A$ and the coalitions of inputs $X_B$ for $B \subseteq A$, one can also be interested in their subsequent interpretation. In the particular case where $\A = \R$, and where the QoI is not random, one natural desirable property would be to interpret this decomposition as shares of QoI.
\begin{defi}[Fractional Möbius decomposition]
Let $G \in \mathcal{M}(E)$ be a model with $E$-valued random inputs $X \sim P_X$, and let $\phi_{P_X}(G)$ be a non-random, non-zero $\R$-valued QoI. Assume that $\phi_{P_X}(G)$ admits a Möbius decomposition (i.e., it can be written as \eqref{eq:mobPsi} with \eqref{eq:mobPhi}). If, $\forall A \in \pset{D}$:
$$\text{sign}\left(\psi_A\right) = \text{sign}\left( \phi_{P_X}(G)\right),$$
Then the Möbius decomposition of $\phi_{P_X}(G)$ is said to be fractional.
\end{defi}
If a Möbius decomposition of $\psi_{P_X}(G)$ is fractional, it ensures that the ratios,
\begin{equation}
     \frac{\psi_A}{\phi_{P_X}(G)},\quad \forall A \in \pset{D},
     \label{eq:ratios}
\end{equation}
are in $[0,1]$, and subsequently that,
$$\sum_{A \in \pset{D}} \frac{\psi_A}{\phi_{P_X}(G)} =1.$$
Essentially, it means that these ratios can be interpreted as shares of QoI attributed to each possible coalition of inputs.

\section{Möbius decompositions for global sensitivity analysis}

The Möbius decompositions defined in \lmeref{thm:mobDecomp} are especially useful in the context of global sensitivity analysis \cite{DaVeiga2021}. In particular, this result allows to:
\begin{itemize}
    \item Show that some existing QoI decompositions proposed in the literature are Möbius decompositions, and actually hold with weaker assumptions on $P_X$ and $G$;
    \item Define decompositions of QoIs being valued in commutative ring with identity other than $\R$.
\end{itemize}

\subsection{Variance decomposition}
Let $X\sim P_X$ be $E$-valued random inputs of an $\R$-valued model $G \in L^2(P_X; \R) \subseteq \mathcal{M}(E)$, and let:
$$\phi_{P_X}(G) = \V{G(X)},$$
be the $\R$-valued QoI, i.e., the variance of the random output $G(X)$.
\begin{proposition}[Variance decomposition]
Let, $\forall A \in \pset{D}$:
$$f_A(X_A) = \E{G(X) \mid X_A},$$
and,
\begin{equation*}
    \varphi_A = \phi_{P_X}(f_A) = \V{\E{G(X) \mid X_A}}.
\end{equation*}
Then, $\V{G(X)}$ admits the following gradual Möbius decomposition:
$$\V{G(X)} = \sum_{A \in \pset{D}} \psi_A,$$
where, $\forall A \in \pset{D}$,
$$ \psi_A = \sum_{B \subseteq A} (-1)^{\diffCard{A}{B}}\V{\E{G(X) \mid X_B}}.$$
Additionally, if the inputs are independent (i.e., $P_X = \prod_{i=1}^d P_{X_i}$), then this decomposition is also fractional.
\label{prop:varDecomp}
\end{proposition}
\begin{proof}[Proof of \propref{prop:varDecomp}]
Since $G \in L^2(P_X; \R)$, one has that, $\forall A \in \pset{D}$:
$$\V{\E{G(X) \mid X_A}} < \infty.$$
Moreover, notice that $\varphi_D = \V{G(X)}$. Applying \lmeref{thm:mobDecomp} ultimately proves the decomposition. Whenever the inputs $X$ are independent, it is well known that $\psi_A \geq 0, \forall A \in \pset{D}$ (see \cite{davgam21}), and since $\V{G(X)} > 0$, the decomposition is thus fractional.
\end{proof}
This result is analogue to the Hoeffding-Sobol' functional analysis-of-variance (FANOVA) \cite{Hoeffding1948, Sobol2001}. Traditionally, this decomposition is the result of a functional decomposition of the model $G$ when it is assumed to be in $L^2(P_X; \R)$, into orthogonal elements, requiring the inputs to be independent. However, as shown above, this decomposition holds even when the inputs are endowed with a dependence structure. However, one can notice that input independence allow the decomposition to be fractional, and hence, in-fine, lets the ratios (i.e., as in \eqref{eq:ratios}) to be interpreted as a percentage of the output's variance attributed to each input coalition.

\subsection{Covariance decomposition}
Now, let $G : E \rightarrow \R^2$ be a model with a bivariate output. Denote $G = \begin{pmatrix}G_1\\ G_2\end{pmatrix}$ and assume that $G \in L^2(P_X; \R^2)$. Let
\begin{align*}
    \phi_{P_X}(G) &= \inprodL{G_1- \E{G_1(X)}, G_2 - \E{G_2(X)} }\\
    &= \Cov{G_1(X)}{G_2(X)},
\end{align*}
in other words, the QoI is the covariance between the two random outputs of the model.
\begin{proposition}[Covariance decomposition]
Let, $\forall A \in \pset{D}$:
$$f_A(X_A) = \begin{pmatrix}
\E{G_1(X) \mid X_A} \\
\E{G_2(X) \mid X_A}
\end{pmatrix},$$
and,
\begin{align*}
    \varphi_A = \phi_{P_X}(f_A) &= \inprodL{\E{G_1(X) \mid X_A} - \E{G_1(X)}, \E{G_2(X) \mid X_A} - \E{G_2(X)} }\\
    &=\Cov{\E{G_1(X) \mid X_A}}{\E{G_2(X) \mid X_A}}
\end{align*}
Then, $\phi_{P_X}(G)$ admits the following gradual Möbius decomposition:
$$\Cov{G_1(X)}{G_2(X)} = \sum_{A \in \pset{D}} \psi_A,$$
where, $\forall A \in \pset{D}$,
$$\psi_A = \sum_{B \subseteq A} (-1)^{\diffCard{A}{B}}\Cov{\E{G_1(X) \mid X_B}}{\E{G_2(X) \mid X_B}}.$$
\label{prop:covarDecomp}
\end{proposition}
\begin{proof}[Proof of \propref{prop:covarDecomp}]
Notice that since $G \in L^2(P_X; \R^2)$, $\forall A \in \pset{D}$, the quantities 
$$\Cov{\E{G_1(X) \mid X_A}}{\E{G_2(X) \mid X_A}}$$
are well defined, and that $\varphi_D =  \Cov{G_1(X)}{G_2(X)}$. Applying \lmeref{thm:mobDecomp} then leads to the gradual decomposition.
\end{proof}

Whenever $G : E \rightarrow \R^k$, for $k \in \N^*$, the two previous results can be generalized using a covariance matrix decomposition (see \cite{Gamboa2013}). Let $\mathcal{D}_k$ be the set of $(k \times k)$ symmetric semi-definite (positive or negative) matrices with non-zero entries on the diagonal, and where elements on the diagonal have the same sign. Note that the triplet $(\mathcal{D}_k, +, \circ)$ where $+$ denotes the usual element-wise matrix addition and $\circ$ denotes the element-wise (Hadamard) multiplication, forms a commutative ring with identity (if all the entries were non-zero, it would be a field since the Hadamard inverse would always be well-defined). Let $\Sigma$ be the covariance matrix of the output $G(X) = \left(G_1(X), \dots, G_k(X)\right)^\top$, defined element-wise, for $i,j = 1, \dots, k$:
$$\Sigma_{ij} =  \Cov{G_i(X)}{G_j(X)}.$$
$\Sigma$ is necessarily semi-definite positive (since it is a covariance matrix) and is in $\mathcal{D}_k$ under the assumption that each element of the output is not constant almost surely. It is then a $\mathcal{D}_k$-valued QoI, and can be decomposed as follows:
\begin{proposition}[Covariance matrix decomposition]
Let, $\forall A \in \pset{D}$, the matrices $\Sigma^A \in \mathcal{D}_k$ be defined element-wise as:
$$\Sigma^A_{i,j} = \Cov{\E{G_i(X) \mid X_A}}{\E{G_j(X) \mid X_A}}, \quad i,j=1,\dots,k.$$
Then, $\Sigma$ admits the following gradual Möbius decomposition:
$$\Sigma = \sum_{A \in \pset{D}} \psi_A,$$
where, $\forall A \in \pset{D}$,
$$\psi_A = \sum_{B \subseteq A} (-1)^{\diffCard{A}{B}}\Sigma^B.$$
\label{prop:covarmatDecomp}
\end{proposition}
\begin{proof}[Proof of \propref{prop:covarmatDecomp}]
Notice that since $G \in L^2(P_X; \R^k)$, $\Sigma^A$ is well-defined $\forall A \in \pset{D}$. Moreover, notice that $\Sigma_D = \Sigma$. Applying \lmeref{thm:mobDecomp} then leads to the decomposition.
\end{proof}
One can notice that, in that setting, decomposing $\Sigma$ amounts to performing the variance decomposition of \propref{prop:varDecomp} on the diagonal elements, and the covariance decomposition of \propref{prop:covarDecomp} on the other elements.

\subsection{Mean maximum-mean discrepancy decomposition}
Aside from moment-based quantities, more complicated QoIs can also be decomposed. Such quantities can be based on kernel embedding of the model $G$. One can refer to \cite{DaVeiga2021} for additional details. For the sake of completeness, some elements are recalled here. 

Let $G \in \mathcal{M}(E)$, be a $Z$-valued model with inputs $X \sim P_X \in \proba(E)$. Denote $P_Y$ the distribution the random output $G(X)$. Moreover, for any $A \in \pset{D}$, let the conditional distribution of $G(X)$ given $X_A$ be denoted by $P_{Y \mid X_A}$. Let $k : Z \times Z \rightarrow \R$ be a kernel associated with a reproducing kernel Hilbert space (RKHS) $\mathcal{H}$ \cite{Berlinet2004}. Let:
$$\mu_G(t) = \int_Z \krnl{z}{t} dP_Y(z) = \int_E \krnl{G(z)}{t} dP_X(z) = \E{\krnl{G(X)}{t}}$$
denotes the \emph{kernel mean embedding} of $G(X)$. Moreover, denote:
$$\mu_{G \mid X}(t) = \E{\krnl{G(X)}{t} \mid X} = \krnl{G(X)}{t}.$$
The \emph{maximum-mean discrepancy} between $P_Y$ and $P_{Y\mid X}$ is given by:
\begin{align*}
    \text{MMD}^2(P_Y, P_{Y \mid X}) &= \normL{\mu_G - \mu_{G \mid X}}{\mathcal{H}}^2\\
    &= \E{\mu_G(G(X))} + \mu_{G \mid X}(G(X)) -2 \E{\krnl{G(X)}{G(X)}}
\end{align*}
One is interested in the QoI defined as the mean MMD, i.e.,
\begin{align*}
    S^{\text{MMD}} &:= \E{\text{MMD}^2(P_Y, P_{Y \mid X})}\\
    &= \E{\mu_G(G(X))} - \E{\krnl{G(X)}{G(X)}}
\end{align*}

\begin{proposition}
Let $X \sim P_X$ be $E$-valued random inputs of a model $G : E \rightarrow Z$. Let $k : Z \times Z \rightarrow \R$ be the reproducing kernel of a RKHS $\mathcal{H}$. Assume that $k$ is such that, $\forall A \in \pset{D}$:
$$\SMMD_A := \mathbb{E}_{X_A}\left[\text{MMD}^2 (P_Y, P_{Y \mid X_A})\right] < \infty.$$
Then, $S^{\text{MMD}}$ admits the following Möbius decomposition:
$$\SMMD = \sum_{A \in \pset{D}} \psi_A,$$
where, $\forall A \in \pset{D}$,
$$\psi_A = \sum_{B \subseteq A} (-1)^{\diffCard{A}{B}}\SMMD_B $$
\end{proposition}
\begin{proof}
By assumption, $\SMMD_A$ is well-defined $\forall A \in \pset{D}$. Moreover, notice that $\SMMD_D = \SMMD$. Applying \lmeref{thm:mobDecomp} then leads to the decomposition.
\end{proof}
This decomposition, analogous to the one presented in \cite{DaVeiga2021}, not only holds when the inputs are independent, but also when they are endowed with a dependence structure.

%%%%%%%%%%%%%%%%%%%%%%%%%%%%%%%%%%%%%%%%%%%%%%%%%%%%%%%%%%%%%%%%%%%%%%%%%%
%% CONCLUSION
\section{Discussion}

Traditionally, in the field of global sensitivity analysis, QoI decompositions are defined using a ``model-centric'' approach. It can be summarized as follows: find a suitable coalitional decomposition of the model $G$ in $L^2$, such that $\phi_{P_X}$ becomes an additive map when applied to $G$. For instance, if the QoI is the variance of the output, orthogonality of the $\psi_A$ is often desired (as defined in \defref{dfi:coalDecomp}). The new viewpoint provided by this communication adopts an ``input-centric'' approach: first define a suitable $\varphi_A$ (as in \eqref{eq:mobPhi}), such that it accurately represents the effect of $X_A$, and then define a suitable decomposition using the reverse implication of the Möbius inversion formula. This approach is analogous to the field of cooperative game theory \cite{Bilbao2000}, where $\varphi$ represents the value function of a cooperative game, and $\psi_A$ are none other than its Harsanyi dividends \cite{Harsanyi1963}. The understanding and possible combination of both approaches to find theoretically suitable candidates for $\varphi$ is the subject of ongoing research.

%%%%%%%%%%%%%%%%%%%%%%%%%%%%%%%%%%%%%%%%%%%%%%%%%%%%%%%%%%%%%%%%%%%%%%%%%%
%% REFERENCES
\bibliographystyle{plain}
\bibliography{references}

\end{document}